\documentclass[a4paper]{amsart}

\usepackage{a4}
\usepackage[T1]{fontenc}

\begin{document}
\newcommand{\J}{\mathcal{J}} 
\newcommand{\cL}{\mathcal{L}}
\newcommand{\I}{\mathcal{I}}
\newcommand{\R}{\mathcal{R}}
\newcommand{\A}{\mathcal{A}}
\newcommand{\F}{\mathcal{F}}
\newcommand{\B}{\mathcal{B}}
\newcommand{\T}{\mathcal{T}}
\newcommand{\C}{\mathcal{C}}
\newcommand{\D}{\mathcal{D}}
\newcommand{\Q}{\mathcal{Q}}
\newcommand{\cS}{\mathcal{S}}
\newcommand{\FF}{\mathbb{F}}
\newcommand{\HH}{\mathbb{H}}
\newcommand{\CC}{\mathbb{C}}
\newcommand{\SSS}{\mathbb{S}}
\newcommand{\TS}{\widetilde{\SSS}}
\newcommand{\RR}{\mathbb{R}}
\newcommand{\OO}{\mathbb{O}}
\newcommand{\TO}{\widetilde{\OO}}
\newcommand{\ZZ}{\mathbb{Z}}
\newcommand{\AAA}{\mathbb{A}}
\newcommand{\ov}{\overline}
\newcommand{\wh}{\widehat}
\newcommand{\e}{\epsilon}
\newtheorem{theorem}{Theorem}[section]
\newtheorem{proposition}[theorem]{Proposition}
\newtheorem{lemma}[theorem]{Lemma}
\newtheorem{corollary}[theorem]{Corollary}

\theoremstyle{remark}
\newtheorem{remark}[theorem]{Remark}
\newtheorem{definition}[theorem]{Definition}
\newtheorem{example}[theorem]{Example}
\newtheorem{examples}[theorem]{Examples}
\newtheorem{question}[theorem]{Question}

\title[Locally complex  and  Cayley-Dickson algebras]
{On locally complex algebras and  low-dimensional Cayley-Dickson algebras }
\thanks{
2010 {\em Math. Subj. Class.} 17A35, 17A45, 17A70, 17D05.}
\thanks{Supported by the Slovenian Research Agency (program No. P1-0288).
}

\author{Matej Bre\v sar, Peter \v Semrl, \v Spela \v Spenko }
\address{Matej Bre\v sar, Faculty of Mathematics and Physics, University of Ljubljana, and Faculty of Natural Sciences and Mathematics, University of Maribor, Slovenia}
\address{Peter \v Semrl and \v Spela \v Spenko, Faculty of Mathematics and Physics, University of Ljubljana, Slovenia}
\email{matej.bresar@fmf.uni-lj.si}
\email{peter.semrl@fmf.uni-lj.si}
\email{spela.spenko@student.fmf.uni-lj.si }

\begin{abstract}The paper begins with 
 short proofs of classical theorems by Frobenius and (resp.) Zorn on  associative and (resp.) alternative real division algebras. These theorems characterize the first three (resp. four) Cayley-Dickson algebras.  
Then we introduce and study  the class of real unital nonassociative algebras in which the subalgebra generated by any  nonscalar element is isomorphic to $\CC$.  We call them {\em locally complex algebras}. In particular, we describe  all such algebras that have  dimension at most $4$. Our main motivation, however, for introducing locally complex algebras is that this concept makes it possible for us to extend  Frobenius' and Zorn's theorems in a way that it also involves the fifth Cayley-Dickson algebra, the sedenions. 
\end{abstract}

\maketitle

\section{Introduction}

The real number field $\RR$, the complex number field $\CC$, and the division agebra of real quaternions $\HH$ are classical examples of associative real division algebras. In 1878 Frobenius \cite{F} proved that in the finite dimensional context they are  also the only examples. Assuming alternativity instead of   associativity,  there is another example: $\OO$, the division algebra of octonions. It turns out that this is the only additional example. This result is attributed to Zorn \cite{Z}.

In Section 3 we  give short and self-contained proofs of these classical theorems by Frobenius and Zorn. Both proofs are based on the same idea. In fact, the proof of Zorn's theorem is a continuation of the proof of Frobenius' theorem. The proofs are constructive, it appears like $\HH$ and $\OO$ are met "unintentionally".  
 
Our proofs of Frobenius' and Zorn's theorems were discovered by accident, when examining the class of real unital algebras with the following  property: the  subalgebra generated by any element different from a scalar multiple of $1$ is isomorphic to $\CC$. These algebras, which we call {\em locally complex}, will be first considered in Section \ref{loccom}. In particular, we will classify 
all locally complex algebras of dimension at most 4.

Unlike real division algebras which exist  only in dimensions $1,2,4$, and $8$ \cite{BM, Ker}, locally complex algebras  exist in abundance in any dimension. However, among alternative (and hence also associative) finite dimensional real algebras, the concepts of division algebras and locally complex algebras  coincide. Frobenius' and Zorn's theorems can be therefore equivalently stated so that one replaces  "division" by "locally complex" in the formulation. This observation paves the way for continuing in the direction of these two theorems.

The algebras $\RR$, $\CC$, $\HH$, and $\OO$ are the first four (real) algebras formed in the Cayley-Dickson process. The next one is the $16$-dimensional algebra $\SSS$ of (real)  {\em sedenions}. It is the first algebra in this process  that is neither a division nor an alternative algebra. Although it is therefore somewhat less attractive than its famous predecessors, $\SSS$ has recently gained a considerable attention.  Over the last years it was considered in several papers by algebraists as well as by mathematical physicists \cite{Baez, Biss, BH, CM, ChanDj, Im, Kuwata, Moreno}. To the best of our knowledge, however, there are no results  that characterize $\SSS$ through its abstract algebraic properties. Moreover, one might get an impression when looking at some of these papers that such characterizations are not really expected (for example, see the introduction in \cite{Biss}).  One of the  goals of this paper is to show that actually they can be established. 

In Section \ref{Secsuper} we consider locally complex algebras that are simultaneously superalgebras with the property that all their homogeneous elements satisfy the alternativity conditions (see \eqref{alt} below). Our main result says that besides the obvious examples, i.e., $\RR$, $\CC$, $\HH$, $\OO$, and $\SSS$, there are exactly two more algebras having these properties, one in dimension $8$ and another one in dimension $16$. As  corollaries we  get three characterizations of $\SSS$: the first one is based on the existence of special elements satisfying a version of the alternativity condition, the second one is based on  the properties of zero divisors, and the third one is based on the structure of subalgebras.  

Let us remark that among the  papers listed above, the one by Calderon and Martin \cite{CM} is philosophically the closest one to our paper since it also considers  superalgebras. However, the two papers  do not seem to have any overlap.  On the other hand, in our final results on sedenions we were influenced by the papers  \cite{Biss, ChanDj, Moreno}.

\section{Preliminaries}\label{Prel}

The purpose of this section is to recall some definitions and elementary properties of the notions needed in subsequent sections. 

Let $A$ be a nonassociative algebra over a field. In this paper we will be actually interested only in the case where this field is $\RR$, although some parts, like the following definitions and comments, make sense in a more general setting. 
 Recall that $A$ is said to be a {\em division algebra} if for every nonzero $a\in A$,  $x\mapsto ax$ and $x\mapsto xa$ are bijective maps from $A$ onto $A$.  If $A$ is finite dimensional, then this is clearly equivalent to the condition that $A$ has no zero divisors. If $A$ is associative, then it is a division algebra if and only if it is unital (i.e., it has a unity $1$) and every nonzero element in $A$ has a multiplicative inverse. For general algebras this is not true.

The real {\em Cayley-Dickson} algebras $\AAA_n$, $n\ge 0$, are (nonassociative) real algebras with involution $\ast$, defined recursively as follows: $\AAA_0 =\RR$ with trivial involution $a^* =a$,
and $\AAA_n$ is the vector space $\AAA_{n-1}\times \AAA_{n-1}$ endowed with  multiplication  and involution  defined by
$$
(a,b)(c,d) = (ac - d^*b, da + bc^*),
$$
$$
(a,b)^* = (a^*,-b).
$$
It is easy to see that $\AAA_n$  is unital (in fact, the unity of $\AAA_n$ is $(1,0)$ where $1$ is the unity of $\AAA_{n-1})$, $x+x^*$ and $xx^*=x^*x$ are scalar multiplies of $1$ for every $x\in \AAA_n$,  and $\dim \AAA_n = 2^n$. Next, it
 is clear that $\AAA_1 = \CC$, and one easily notices that $\AAA_2 =\HH$, the {\em quaternions}.  The next algebra in this process is $\AAA_3 =\OO$, the {\em octonions}. For an excellent survey on octonions we refer the reader to \cite{Baez}. Let us record here just a few basic properties of $\OO$. First of all, $\OO$ is
  an $8$-dimensional division algebra. Denoting its basis by $\{1,e_1,\ldots,e_7\}$, the multiplication in $\OO$ is  determined by the following table:
\begin{small} 
\begin{center}
\begin{tabular}{|r|r|r|r|r|r|r|r|}
\hline
  & $e_1$ & $e_2$ & $e_3$ & $e_4$ & $e_5$ & $e_6$ & $e_7$\\\hline
  $e_1$ & $-1$ & $e_3$ & $-e_2$ & $e_5$ & $-e_4$ & $-e_7$ & $e_6$\\\hline
   $e_2$ & $-e_3$ & $-1$ & $e_1$ & $e_6$ & $e_7$ & $-e_4$ & $-e_5$\\\hline
     $e_3$ & $e_2$ & $-e_1$ & $-1$ & $e_7$ & $-e_6$ & $e_5$ & $-e_4$\\\hline
       $e_4$ & $-e_5$ & $-e_6$ & $-e_7$ & $-1$ & $e_1$ & $e_2$ & $e_3$\\\hline
         $e_5$ & $e_4$ & $-e_7$ & $e_6$ & $-e_1$ & $-1$ & $-e_3$ & $e_2$\\\hline
           $e_6$ & $e_7$ & $e_4$ & $-e_5$ & $-e_2$ & $e_3$ & $-1$ & $-e_1$\\\hline
             $ e_7$  & $-e_6$ & $e_5$ & $e_4$ & $-e_3$ & $-e_2$ & $e_1$ & $-1$\\\hline
\end{tabular}
\end{center}
\end{small}
Note that the linear span of $1,e_1,e_2,e_3$ is a subalgebra of $\OO$ isomorphic to $\HH$.

It is well known that $\OO$ is a division algebra which is not associative. However, it is "almost" associative - namely, it is  alternative. Recall that an algebra $A$ is said to be {\em alternative} if 
\begin{equation}\label{alt}
x^2 y = x (xy)\quad\mbox{and}\quad yx^2 = (yx)x
\end{equation}
holds for all $x,y\in A$. Incidentally, Artin's theorem says that this is equivalent to the condition that any two elements generate an  associative subalgebra \cite[p. 36]{ZSSS}. We shall need the identities from \eqref{alt} in their linearized forms:
\begin{equation} \label{eq2}
(xz+zx)y=x(zy)+z(xy),\,\, y(xz+zx)=(yx)z+(yz)x.
\end{equation}
Let us also record the so-called middle Moufang identity which, as one easily checks (see, e.g., \cite[p. 35]{ZSSS}),  holds in  every alternative algebra:
\begin{equation} \label{eq3}
(xy)(zx) = x(yz)x.
\end{equation}
With regard to the right-hand side of \eqref{eq3} it should be pointed out that alternative algebras are flexible, i.e., $x(yx) = (xy)x$ holds (after all, this follows from Artin's theorem), and therefore there is a convention to write $xyx$ instead of $(xy)x$  or $x(yx)$. 

The next algebra obtained by the Cayley-Dickson process is the $16$-dimensional algebra $\AAA_4=\SSS$, the {\em sedenions}. Let $\{1,e_1,\ldots,e_{15}\}$ be a basis of $\SSS$. This is the multiplication table for $\SSS$:

\begin{center}
\begin{tiny}
\begin{tabular}{|r|r|r|r|r|r|r|r|r|r|r|r|r|r|r|r|r|}
\hline
& $e_1$ & $e_2$ & $e_3$ & $e_4$ & $e_5$ & $e_6$ & $e_7$ & $e_8$ & $e_9$ & $e_{10}$ & $e_{11}$ & $e_{12}$ & $e_{13}$ & $e_{14}$ & $e_{15}$\\\hline
 $e_1$ & $-1$ & $e_3$ & $-e_2$ & $e_5$ & $-e_4$ & $-e_7$ & $e_6$ & $e_9$ & $-e_8$ & $-e_{11}$ & $e_{10}$ & $-e_{13}$ & $e_{12}$ & $e_{15}$ & $-e_{14}$\\\hline
  $e_2$ & $-e_3$ & $-1$ & $e_1$ & $e_6$ & $e_7$ & $-e_4$ & $-e_5$ & $e_{10}$ & $e_{11}$ & $-e_8$ & $-e_9$ & $-e_{14}$ & $-e_{15}$ & $e_{12}$ & $e_{13}$\\\hline
   $e_3$ & $e_2$ & $-e_1$ & $-1$ & $e_7$ & $-e_6$ & $e_5$ & $-e_4$ & $e_{11}$ & $-e_{10}$ & $e_9$ & $-e_8$ & $-e_{15}$& $e_{14}$ & $-e_{13}$ &$e_{12}$\\\hline
    $e_4$ &$-e_5$ & $-e_6$ & $-e_7$ & $-1$ & $e_1$ & $e_2$ & $e_3$ & $e_{12}$ & $e_{13}$ & $e_{14}$ &$e_{15}$ & $-e_8$ & $-e_9$ & $-e_{10}$ & $-e_{11}$\\\hline
     $e_5$ & $e_4$ & $-e_7$ & $e_6$ &$-e_1$ & $-1$ & $-e_3$ & $e_2$ & $e_{13}$ & $-e_{12}$ & $e_{15}$ & $-e_{14}$ & $e_9$ &$-e_8$ & $e_{11}$ & $-e_{10}$\\\hline
      $e_6$ & $e_7$ & $e_4$ & $-e_5$ & $-e_2$ & $e_3$ & $-1$ &$-e_1$ & $e_{14}$ & $-e_{15}$ & $-e_{12}$ & $e_{13}$ & $e_{10}$ & $-e_{11}$ & $-e_8$ & $e_9$\\\hline
       $e_7$ & $-e_6$ & $e_5$ & $e_4$ & $-e_3$ & $-e_2$ & $e_1$ & $-1$ & $e_{15}$ &$e_{14}$ & $-e_{13}$ & $-e_{12}$ & $e_{11}$ & $e_{10}$ & $-e_9$ & $-e_8$\\\hline
        $e_8$ & $-e_9$ &$-e_{10}$ & $-e_{11}$ & $-e_{12}$ & $-e_{13}$ & $-e_{14}$ & $-e_{15}$ & $-1$ & $e_1$ & $e_2$ &$e_3$ & $e_4$ & $e_5$ & $e_6$ & $e_7$\\\hline
         $e_9$ & $e_8$ & $-e_{11}$ & $e_{10}$ & $-e_{13}$ &$e_{12}$ & $e_{15}$ & $-e_{14}$ & $-e_1$ & $-1$ & $-e_3$ & $e_2$ & $-e_5$ & $e_4$ & $e_7$ & $-e_6$\\\hline
          $e_{10}$ & $e_{11}$ & $e_8$ & $-e_9$ & $-e_{14}$ & $-e_{15}$ & $e_{12}$ & $e_{13}$ &$-e_2$ & $e_3$ & $-1$ & $-e_1$ & $-e_6$ & $-e_7$ & $e_4$ & $e_5$\\\hline
           $e_{11}$ &$-e_{10}$ & $e_9$ & $e_8$ & $-e_{15}$ & $e_{14}$ & $-e_{13}$ & $e_{12}$ & $-e_3$ & $-e_2$ &$e_1$ & $-1$ & $-e_7$ & $e_6$ & $-e_5$ & $e_4$\\\hline
            $e_{12}$ & $e_{13}$ & $e_{14}$ & $e_{15}$ &$e_8$ & $-e_9$ & $-e_{10}$ & $-e_{11}$ & $-e_4$ & $e_5$ & $e_6$ & $e_7$ & $-1$ & $-e_1$ &$-e_2$ & $-e_3$\\\hline
             $e_{13}$ & $-e_{12}$ & $e_{15}$ & $-e_{14}$ & $e_9$ & $e_8$ & $e_{11}$ & $-e_{10}$ & $-e_5$ & $-e_4$ & $e_7$ & $-e_6$ & $e_1$ & $-1$ & $e_3$ & $-e_2$\\\hline
              $e_{14}$ & $-e_{15}$ & $-e_{12}$ & $e_{13}$ & $e_{10}$ & $-e_{11}$ & $e_8$ & $e_9$ &$-e_6$ & $-e_7$ & $-e_4$ & $e_5$ & $e_2$ & $-e_3$ & $-1$ & $e_1$\\\hline
               $e_{15}$ & $e_{14}$ &$-e_{13}$ & $-e_{12}$ & $e_{11}$ & $e_{10}$ & $-e_9$ & $e_8$ & $-e_7$ & $e_6$ & $-e_5$ & $-e_4$ & $e_3$ & $e_2$ & $-e_1$& $-1$\\\hline
\end{tabular}
\end{tiny}
\end{center}
The sedenions have zero divisors and they are not an alternative algebra.  Anyhow, we shall see that they are close enough to alternative division algebras, so that these approximate properties are "almost" characteristic for $\SSS$. Let us recall the definition of another notion needed for dealing with these properties.

An algebra $A$ is said to be a {\em superalgebra} if it is $\ZZ_2$-graded, i.e., there exist linear subspaces $A_i$, $i\in \mathbb Z_2$, such that $A=A_0\oplus A_1$ and $A_iA_j\subseteq A_{i+j}$ for all $i,j\in \ZZ_2$. We call $A_0$ an {\em even} and $A_1$ an {\em odd} part of $A$.  Elements in $A_0\cup A_1$ are said to be {\em homogeneous}. Note that if $A$ is unital, then $1\in A_0$.

Cayley-Dickson algebras possess a natural superalgebra structure. Indeed,  $A=\AAA_n$ becomes a superalgebra by defining $A_0 = \AAA_{n-1}\times 0$ and $A_1 = 0\times \AAA_{n-1}$. This simple observation is the concept behind the contents of Section \ref{Secsuper}. 

The algebras $\AAA_n$, $n\ge 4$, are not alternative, but at least they have certain nonscalar elements that share many properties with elements in alternative algebras: these are scalar multiples of the element $e=(0,1)$, where $1$ is of course the unity of $\AAA_{n-1}$ (see e.g. \cite[Section 5]{Biss}).  Let us point out only one property that is sufficient for our purposes: $e$ satisfies $x^2 e = x(xe)$ for all $x\in\AAA_n$. This can be easily verified. Moreover, this property is "almost" characteristic for $e$: only elements in the linear span of $1$ and $e$ satisfy this identity for every $x$ \cite[Lemma 1.2]{AS} (the authors are thankful to Alberto Elduque for drawing their attention to this result).
Now, let us call an element $a$ in an arbitrary nonassociative algebra $A$ an {\em alter-scalar} if $a$ is not a scalar and satisfies   $x^2 a = x(xa)$ holds for all $x\in A$. (A similar, but not exactly the same notion of a strongly alternative element was defined in \cite{Moreno2}. There is also a standard notion of an alternative element defined through the condition $a^2x = a(ax)$ for every $x$, but this is too weak for our goals). What is important for us is that $\SSS$ contains   alter-scalars. With respect to the notation introduced above, these are nonzero scalar multiplies of  $e_8$.  Thus, the standard basis of $\SSS$ has an element that is in some sense "better" than the others. This does not seem to be the case with the preceding Cayley-Dickson algebras.

 Next we recall that an algebra 
$A$ is said to be {\em quadratic} if it is unital and  the elements $1,x,x^2$ are linearly dependent for every $x\in A$. 
Thus, for every
$x\in A$ there exist $t(x),n(x)  \in \RR$ such that $x^{2} - t(x)x + n(x)=0$.  Obviously, $t(x)$ and $n(x)$ are uniquely determined if $x\notin \RR$. Setting $t(\lambda) = 2\lambda$ and $n(\lambda)= \lambda^2$ for $\lambda\in\RR$, we can then consider $t$ and $n$ as maps from $A$ into $\RR$ (the reason for this definition is that in this way $t$ becomes a linear functional, but we shall not need this).  We call $t(x)$ and $n(x)$  the {\em trace}  and the {\em norm} of $x$, respectively. For some elementary properties of quadratic algebras, a characterization of quadratic alternative algebras, and further references  we refer to \cite{Eld}.

From $x^2 - (x+x^*)x + x^*x =0$ we see that all algebras $\AAA_n$ are quadratic. Further,
every  real division algebra $A$ that is algebraic and power-associative (this means that every subalgebra generated by one element is associative) is automatically quadratic. Indeed, if $x\in A$  then there exists a nonzero polynomial $f(X)\in\RR[X]$ such that $f(x)=0$.
Writing $f(X)$ as the product of linear and quadratic polynomials in $\RR[X]$  it follows that $p(x)=0$ for some $p(X)\in\RR[X]$ of degree $1$ or $2$.  In particular, algebraic alternative (and hence associative) real division algebras are quadratic.

Finally, if  $A$ is a real unital  algebra, i.e., an algebra over $\RR$ with unity $1$, then we    
shall follow a standard convention and identify $\RR$ with $\RR 1$; thus we shall write $\lambda$ for $\lambda 1$, where $\lambda\in\RR$.

\section{Frobenius' and Zorn's theorems}

Our first lemma is well known. It describes one of the basic properties of quadratic algebras.  We give the proof for the sake of completness. 

\begin{lemma} \label{L0}
Let $A$ be a quadratic real algebra. Then $U=\{u\in A\setminus{\RR}\,|\,u^2 \in \RR\}\cup\{0\}$ is a linear subspace of $A$,  $uv+vu\in \RR$ for all $u,v\in U$, and $A=\RR\oplus U$. 
\end{lemma}

\begin{proof} Obviously, $U$ is closed under scalar multiplication.  We have to show that $u,v\in U$ implies $u+v\in U$. If $u,v,1$ are linearly dependent, then one easily notices that  already $u$ and $v$ are dependent, and the result follows. Thus, let $u,v,1$ be independent. We have $(u+v)^2 + (u-v)^2=2u^2 + 2v^2 \in \RR$. On the other hand, as $A$ is quadratic there exist
$\lambda,\mu\in\RR$ such that $(u+v)^2 - \lambda(u+v)\in\RR$ and $(u-v)^2 - \mu(u-v)\in\RR$, and hence
$ \lambda(u+v) + \mu(u-v)\in \RR$.  However, the independence of $1,u,v$ implies $\lambda + \mu = \lambda -\mu =0$, so that $\lambda=\mu=0$. This  proves that   $u\pm v\in U$. Thus $U$ is indeed a subspace of $A$. Accordingly,
$u v + v u = (u+v)^{2} - u^{2} - v^{2}\in\RR$ for all $u,v\in U$. Finally, 
if $a\in A\setminus{\RR}$, then $a^2 - \nu a\in\RR$ for some $\nu\in \RR$, and therefore $u=a -\frac{\nu}{2}\in U$; thus, $a =  \frac{\nu}{2} + u\in \RR\oplus U$. 
\end{proof}

\begin{remark} \label{Rem1}
If $A$ is additionally a division algebra, then every nonzero $u\in U$ can be written as $u=\alpha v$ with
$\alpha\in \RR$ and $v^2 =-1$. Indeed, since $u^2\in\RR$ and since $u^2$ cannot be $\ge 0$  (otherwise $(u-\alpha)(u+\alpha) = u^2 - \alpha^2$ would be $0$ for some $\alpha\in\RR$) we have $u^2=- \alpha^2$ with $0\ne \alpha\in\RR$. Thus, $v = \alpha^{-1}u$ is a desired element.
\end{remark}

Note that
by $\langle u,v\rangle = - \frac{1}{2}(uv +vu)$ one defines an inner product on $U$ if $A$ is a division algebra. The next lemma therefore deals with nothing but the  Gram-Schmidt process. Nevertheless, we give the proof.

\begin{lemma} \label{L1}
Let $A$ be a  quadratic real division algebra, and let $U$ be as in Lemma \ref{L0}. Suppose  $e_1,\ldots,e_k\in U$ are such that $e_i^{2} = -1$  for all $i\le k$ and $e_ie_j =- e_j e_i$ for all  $i,j\le k$, $i\ne j$. If $U$ is not equal to the linear span of $e_1,\ldots,e_k$, then there exists $e_{k+1}\in U$ such that $e_{k+1}^{2} = -1$ and $e_i e_{k+1} =- e_{k+1} e_i$ for all $i\le k$.
\end{lemma}

\begin{proof}
 Pick $u\in U$ that is not contained in the linear span of $e_1,\ldots,e_k$, and set 
$\alpha_i = \frac{1}{2}(ue_i + e_i u) \in \RR$ (by Lemma \ref{L0}). Note that 
$v = u + \alpha_1 e_1+\ldots +\alpha_k e_k$ satisfies $e_i v = -v e_i$ for all $i\le k$. Let $e_{k+1}$ be a scalar multiple of $v$ such that $e_{k+1}^{2} = -1$ (Remark \ref{Rem1}). Then $e_{k+1}$ has all desired properties. 
\end{proof}

\begin{theorem} \label{TF}
{\bf (Frobenius' theorem)}  An algebraic  associative real division algebra $A$ is isomorphic to $\RR$, $\CC$, or $\HH$.
\end{theorem}

\begin{proof} As pointed out at the end of Section \ref{Prel}, $A$ is quadratic. We may assume that $n=\dim A\ge 2$. By Remark \ref{Rem1} we can
fix $i\in A$ such that $i^2=-1$. Thus, $A\cong\CC$ if $n=2$.  Let $n > 2$. By Lemma \ref{L1} there is $j\in A$ such that $j^2=-1$ and $ij=-ji$. Set $k=ij$. Now one immediately checks that $k^2 =-1$, $ki = j =-ik$, $jk=i=-kj$, and  $i,j,k$ are linearly independent. Therefore $A$ contains a subalgebra isomorphic to $\HH$. It remains to show that $n$ is not $>4$. If it was, then by Lemma \ref{L1} there would exist
$e\in A$ such that $e\ne 0$, $ei = -ie$, $ej = - je$, and $ek=-ke$. However, from the first two identities we infer $ eij= -iej=ije$; since $ij=k$, this contradicts the third identity.
\end{proof}

In standard graduate algebra textbooks one can find different proofs of Frobenius' theorem. In some of them the advanced theory is used, but there are also such that use only elementary tools, e.g.,  \cite{Her} and \cite{Lam}. The proof in \cite{Her} is actually based on similar ideas than our proof, but it is considerably lengthier. The one in \cite{Lam}
(which is based on \cite{Pal}) is different, and also  short. 

We believe that our proof, consisting of four simple steps (Lemma \ref{L0}, Remark \ref{Rem1}, Lemma \ref{L1}, and the final proof), should be easily understandable to  undergraduate students. Some of these steps, especially both lemmas, are of independent interest.

We now switch to the proof of Zorn's theorem. We  need a simple lemma:

\begin{lemma} \label{L2}
Let $A$ be an alternative algebra, and let $e_1,\ldots,e_k\in A$ be such that $e_ie_j\in\{e_1,\ldots,e_k\}$ whenever $i\ne j$. If $w\in A$ is such that  $e_iw=-we_i$ for every $i$, then $(e_ie_j)w=-e_i(e_jw)$ and $w(e_ie_j)=-(we_i)e_j$ whenever $i\neq j$.
 \end{lemma}

\begin{proof} Just set $x=e_i$, $y=e_j$, and $z=w$  in \eqref{eq2}, and the result follows.
\end{proof}

\begin{theorem} \label{TZ}
{\bf (Zorn's theorem)} An algebraic  alternative real division algebra $A$ is isomorphic to $\RR$, $\CC$, $\HH$, or $\OO$.
\end{theorem}
\begin{proof}
Since a subalgebra generated by two elements is associative, the first part of the proof of Theorem \ref{TF} remains unchanged in the present context. We may therefore assume that $A$ contains a copy of $\HH$ and that $n=\dim A > 4$. Let us just change the notation and write $e_1 = i$, $e_2 = j$, and $e_3=k$. 
By Lemma \ref{L1} there exists $e_4\in A$ such that $e_4^2=-1$ and $e_4e_i=-e_ie_4$ for $i=1,2,3$. 
 Now define $e_5=e_1e_4$, $e_6=e_2e_4$, $e_7=e_3e_4$. Using the alternativity and anticommutativity relations we see that 
$$
 e_5^2=e_6^2=e_7^2=-1,
 $$
$$
  e_1e_5 = -e_5e_1 = e_2e_6=-e_6e_2 = e_3e_7=-e_7e_3= -e_4,
$$
$$
e_4e_5 = -e_5e_4 = e_1,\,\, e_4e_6 = -e_6e_4 = e_2,\,\, e_4 e_7 = -e_7e_4 = e_3.
$$
Further, using \eqref{eq3} we  obtain
$$
e_5e_6 = -e_6e_5 = -e_3 ,\,\, e_6e_7 = -e_7e_6 = -e_1,\,\, e_7e_5 = -e_5e_7 = -e_2.
$$
 Finally,  use Lemma \ref{L2} with $k=3$ and $w=e_4$, and note that the resulting identites  yield the rest of the multiplication table. 
 
 It is easy to see that $1,e_1,\ldots,e_7$ are linearly independent. Indeed, by taking squares we first see that $\sum_{i=1}^7 \lambda_ie_i$ cannot be a nonzero scalar; if $\sum_{i=1}^7 \lambda_ie_i =0$, then after multiplying this relation with $e_i$ we get $\lambda_i=0$. Thus, we have showed that $A$ contains $\OO$.

It remains to show that $n=8$. Suppose $n > 8$.  Then, by Lemma \ref{L1}, there exists $f\in A$ such that $f\ne 0$ and $fe_i=-e_if$, $1\le i\le 7$. Lemma \ref{L2} tells us that $f$ also satisfies $(e_ie_j)f=-e_i(e_jf)$ and $f(e_ie_j)=-(fe_i)e_j$ for $i\neq j$. Accordingly, 
\begin{equation}
\label{espe}
 e_1(e_2(e_4f))=-e_1((e_2e_4)f)= -e_1(e_6f) = (e_1e_6)f =- e_7f.
\end{equation}
Note that for  $1\le i\le 3$ we have
$$
e_i(e_4f)=-(e_ie_4)f=f(e_ie_4)=-f(e_4e_i)=(fe_4)e_i=-(e_4f)e_i.
$$
This makes it possible for us to apply 
Lemma \ref{L2} for $k=3$ and $w=e_4f$.  In particular this gives $(e_1e_2)(e_4f)=-e_1(e_2(e_4f))$. Consequently,
$$
e_1(e_2(e_4f)) = -e_3(e_4f) = (e_3e_4)f= e_7f,
$$
 contradicting \eqref{espe}. 
\end{proof}

\begin{remark} \label{Rdalje}
From the first part of the proof we see that if an alternative (not necessarily a division) real algebra $A$ contains a copy of $\HH$ and $\dim A > 4$, then it also contains a copy of $\OO$.
\end{remark}

Classical versions of Frobenius' and Zorn's theorems deal with finite dimensional algebras rather than with (slightly more general) algebraic ones. Our method, however, yields these  more general versions for free. But actually we shall need the more general version of Zorn's theorem in Section \ref{Secsuper}.
  
We cannot claim that any of the arguments given in this section is entirely original. After finding these proofs we have realized, when searching the literature, that many of these ideas appear in different texts. But to the best of our knowledge nobody has compiled these arguments in the same way that leads to short and direct proofs of theorems by Frobenius and Zorn. Therefore we hope and believe that this section is of some value.

\section{Locally complex algebras} \label{loccom}

As already mentioned, we define a {\em locally complex algebra} as  a real unital algebra $A$  such that every $a\in A\setminus{\RR}$ generates a subalgebra isomorphic to $\CC$.  
  A locally complex algebra $A$ is obviously quadratic. We can therefore consider the trace $t(a)$ and the norm $n(a)$ of each $a\in A$.

\begin{lemma}\label{Tlc} 
The following conditions are equivalent for a real unital algebra $A$:
\begin{enumerate}
\item[(i)] $A$ is locally complex;
\item[(ii)] every $0\ne a\in A$ has a multiplicative inverse lying in $\RR a + \RR$;
\item[(iii)] $A$ is quadratic and $A$ has no nontrivial idempotents  or square-zero elements;
\item[(iv)] $A$ is quadratic and  $n(a)> 0$ for every $0\ne a\in A$.
\end{enumerate}
Moreover, if $2\le \dim A = n < \infty$, then {\rm (i)-(iv)} are equivalent to 
\begin{enumerate}
\item[(v)] $A$ has a basis $\{1,e_1,\ldots, e_{n-1}\}$ such that $e_i^{2} = -1$  for all $i$ and $e_i e_j =- e_j e_i$ for all $i\ne j$. 
\end{enumerate}
\end{lemma}

\begin{proof}
It is easy to see that  (i)$\Longrightarrow$ (ii) and (ii)$\Longrightarrow$ (iii). Suppose $A$ is quadratic and $n(a) \le 0$ for some $0\ne a\in A$. Then $a\notin\RR$. Therefore also $b=a - \frac{t(a)}{2} \notin \RR$. Note that
$b^2 \ge 0$. If $b^2 =0$, then $A$ has a nontrivial nilpotent. If $b^2 > 0$, i.e., $b^2 = \alpha^2$ for some $0\ne \alpha\in\RR$, then $e=\frac{1}{2}(1-\alpha^{-1}b)$ is a nontrivial idempotent in $A$. 
Thus, (iii)$\Longrightarrow$ (iv). The proof of (iv)$\Longrightarrow$ (ii) is also straightforward. Therefore (ii)-(iv) are equivalent.  Now assume (ii)-(iv) and pick $a\in A\setminus{\RR}$. Then $b= a - \frac{t(a)}{2}$ satisfies $b^{2} \in\RR$. Just as in the argument above we see that 
  $b^{2}$ cannot be  $\ge 0$. Hence $b^{2} = -\alpha^2$ for some $\alpha\in\RR\setminus{\{0\}}$, and so $i = \alpha^{-1}b$ satisfies  $i^{2} = -1$. This yields (i).
  
  Finally, assume $2\le \dim A = n < \infty$. The implication (i)-(iv) $\Longrightarrow$ (v) follows from (the proof of) Lemma \ref{L1}. Assuming (v) and writing $a\in
   A$ as $a = \lambda_0 +\sum_{i=1}^{n-1}\lambda_i e_i$, we see that  $a^{2} - t(a)a + n(a)=0$ with $t(a)=2\lambda_0$ and $n(a) = \sum_{i=0}^{n-1}\lambda_i^2$. Thus, (iv) holds.
  \end{proof}

We can now list various examples of locally complex algebras.

\begin{example} \label{ex0}
A quadratic real division algebra is locally complex.
 \end{example}
 
\begin{example} \label{ex1}
Let $J_n$ be an $n$-dimensional real vector space, and let $\{1,e_1,\ldots,e_{n-1}\}$ be its basis. Define a multiplication in $J_n$ so that $1$ is of course the unity, and the others are multiplied according to $e_i e_j = -\delta_{ij}$. Then $J_n$ is a locally complex algebra and simultaneously  a Jordan algebra. Another way of representing $J_n$ is by identifying it with $\RR\times \RR^{n-1}$, and defining multiplication by $(\lambda,u) (\mu , v) = (\lambda\mu - \langle u,v\rangle, \lambda v + \mu u)$, where 
$\langle \,.\,,\,.\,\rangle$ denotes the standard inner product on $\RR^{n-1}$.
 \end{example}

\begin{example} \label{ex2}
A  real unital  algebra $A$ is said to be {\em nicely normed} if there exists a linear map $\ast:A\to A$ such that $a^{**} =a$, $(a b)^* = b^* a^*$ for all $a,b\in A$, and $a+a^*\in\RR$, $a a^* = a^* a > 0$ for all $0\ne a\in A$ (cf. \cite[p.\,154]{Baez}). These algebras form an important subclass of locally complex algebras. Namely,
every element $a$ in such an algebra $A$ satisfies $a^{2} - t(a)a + n(a)=0$ with $t(a)=a+a^*$ and $n(a) = a a^*$, so that $A$ is indeed locally complex. Note that $U=\{u\in A\setminus{\RR}\,|\,u^2 \in \RR\}\cup\{0\} = \{u\in A\,|\,u^* = -u\}$.

In particular, the Cayley-Dickson algebras $\AAA_n$ are  nicely normed, and hence locally complex.
 \end{example}

From Lemma \ref{Tlc} we can deduce the following characterization of finite dimensional nicely normed algebras.

\begin{corollary} \label{Cnn}
let $A$ be a real unital algebra. If $2\le \dim A = n < \infty$, then the following conditions are equivalent:
\begin{enumerate}
\item[(i)] $A$ is nicely normed;
\item[(ii)] $A$ has a basis $\{1,e_1,\ldots, e_{n-1}\}$ such that $e_i^{2} = -1$  for all $i$ and $e_i e_j =- e_j e_i \in {\rm span}\{e_1,\ldots,e_{n-1}\}$ for all $i\ne j$. 
\end{enumerate}
\end{corollary}

\begin{proof}
Assume (i). By Lemma  \ref{Tlc}\,(v) $A$ has a basis $\{1,e_1,\ldots, e_{n-1}\}$ that has all desired properties except that we do not know yet that $e_i e_j  \in {\rm span}\{e_1,\ldots,e_{n-1}\}$. In view of the  observation in Example \ref{ex2}  we have  ${\rm span}\{e_1,\ldots,e_{n-1}\} = U = \{u\in A\,|\,u^* = -u\}$. Therefore, if $i\ne j$, $(e_i e_j)^\ast = e_j^\ast e_i^\ast = e_j e_i = -e_i e_j$, and hence  $e_i e_j\in U$. Conversely, if (ii) holds, then we can define $\ast$ according to $1^* = 1$ and $e_{i}^* = -e_{i}$, and one easily checks that this makes $A$ a nicely normed algebra.
\end{proof}

If $A$ is a {\em commutative} finite dimensional locally complex algebra, then the $e_i$'s from (v) in  Lemma  \ref{Tlc} must satisfy $e_ie_j = 0$ if $i\ne j$. This can be interpreted as follows.

\begin{corollary}\label{Cjor}
Let  $A$ be a locally complex algebra with $2\le \dim A = n < \infty$. Then $A$ is commutative if and only if $A\cong J_n$.
\end{corollary}

Let $A$ be an alternative  real algebra. If  $A$ is an algebraic division algebra, then it is quadratic, and hence , as already mentioned, locally complex. Conversely,  if $A$ is locally complex, then by Lemma \ref{Tlc}\,(ii) for every $0\ne a\in A$ there exist $\lambda,\mu\in \RR$ such that $a(\lambda a + \mu) =1$. Since $A$ is alternative it follows that for every $y\in A$ the equation $ax=y$ has the solution $x= (\lambda a+ \mu)y$. Similarly one solves the equation $xa = y$. Therefore $A$ is an algebraic division algebra.  Accordingly, Frobenius' and Zorn's theorem can be equivalently stated as follows.

\begin{theorem}\label{TFZ}
{\bf (Frobenius' and Zorn's theorems)}   An  associative locally complex algebra is isomorphic to $\RR$, $\CC$, or $\HH$.  An  alternative  locally complex algebra is isomorphic  to $\RR$, $\CC$, $\HH$, or $\OO$.
\end{theorem}

As already mentioned in the introduction, this version of Frobenius' and Zorn's theorems indicates the direction in which 
these theorems can be generalized. We shall deal with this in the next section. 

In the rest of this section 
we  will classify  locally complex algebras up to dimesion 4. 
Clearly, $\RR$ and $\CC$ are, up to an isomorphism, the only locally 
complex algebras of dimension $\le 2$. 

We fix some notation.
The members of $\RR \times \RR^2$
will be denoted by $(\lambda ,x) = (\lambda , x_1 , x_2)$ and the members of $\RR \times \RR^3$
by $(\lambda ,x) = (\lambda, x_1 , x_2 , x_3)$. For each (ordered) pair $x,y \in \RR^2$
we denote by $|x\ y|$ the $2\times 2$ determinant $ \left| \begin{array}{ccc} x_1 & y_1 \\
x_2 & y_2 \end{array} \right|$. The symbol $x\times y$ stands for the usual vector
product (cross product)
of $x,y \in \RR^3$, while $(x,y,z)$ denotes the scalar triple product
$(x,y,z) = \langle x\times y , z \rangle$, $x,y,z \in \RR^3$. 

Let $t,s$ be nonnegative real numbers. We denote by $A_{t,s}$ the 3-dimensional
algebra $A_{t,s}= \RR \times \RR^2$ with the multiplication given by
$$
(\lambda , x) \, (\mu , y) = (\lambda \mu - \langle x,y \rangle + t |x\ y| ,\lambda y
+ \mu x + s |x\ y|e_1 ),
$$
where $e_1 = (1,0)\in \RR^2$. It follows from Lemma \ref{Tlc}\,(v) that $A_{t,s}$ is
a locally complex algebra. We will show that each 3-dimensional locally complex
algebra $A$ is isomorphic to $A_{t,s}$ for some $(t,s)\in [0,\infty) \times [0, \infty)$
and that $A_{t,s}$ and $A_{t', s'}$ are not isomorphic whenever $(t,s)\not= (t' , s')$.
In short, we have the following classification theorem for 3-dimensional
locally complex algebras.

\begin{theorem}\label{bv1}
The map $(t,s) \mapsto A_{t,s}$, $t,s \ge 0$, induces a bijection between
$[0,\infty) \times [0, \infty)$ and isomorphism classes of 3-dimensional
locally complex algebras.
\end{theorem}

\begin{proof}
We first show that each 3-dimensional locally complex
algebra $A$ is isomorphic to $A_{t,s}$ for some $(t,s)\in [0,\infty) \times [0, \infty)$.
It is a straightforward consequence of Lemma \ref{Tlc}\,(v) that $A$ is isomorphic
to $\RR \times \RR^2$ with the multiplication given by
$$
(\lambda , x) \, (\mu , y) = (\lambda \mu - \langle x,y \rangle ,\lambda y
+ \mu x ) + |x\ y| (t,z)
$$
for some $(t,z)\in \RR \times \RR^2$.  So, we may, and we will assume that
$A$ is this algebra. 
We have two possibilities; either $t\ge 0$, or $t<0$. Let us  consider only the second one; the case when $t\ge 0$ can be handled in a similar, but simpler way.
Set $s= \| z\|$.
There exists an orthogonal $2\times 2$ matrix $Q$
such that $Qz=-se_1$ and $\det Q = -1$. Observe that $|Qx \ Qy| = (\det Q) |x\ y| = - |x\ y|$
and $\langle Qx , Qy \rangle = \langle x,y \rangle$, $x,y \in \RR^2$.
We claim that the map $\varphi : A \to A_{|t|,s}$ given by $\varphi (\lambda , x) =
(\lambda , Qx)$, $(\lambda , x) \in \RR \times \RR^2$, is an isomorphism. Clearly, it
is linear and bijective. Moreover, we have
$$
\varphi( (\lambda, x)\, (\mu, y)) = \varphi ((\lambda \mu - \langle x,y \rangle 
+ t|x\ y|,\lambda y
+ \mu x + |x\ y|z ))
$$
$$
= (\lambda \mu - \langle x,y \rangle 
+ t|x\ y|,\lambda Qy
+ \mu Qx - s|x\ y|e_1 ).
$$
On the other hand,
$$
\varphi(\lambda, x)\, \varphi(\mu, y)= (\lambda , Qx)\, (\mu, Qy) 
$$
$$
= (\lambda \mu - \langle Qx,Qy \rangle + |t|\ |Qx\ Qy| ,\lambda Qy
+ \mu Qx + s |Qx\ Qy|e_1 ) 
$$ $$
= (\lambda \mu - \langle x,y \rangle + t |x\ y| ,\lambda Qy
+ \mu Qx - s |x\ y|e_1 ).
$$
Hence, $\varphi$ is an isomorphism. It remains to show that if
$A_{t,s}$ and $A_{t', s'}$ are isomorphic for some $(t,s),(t' , s')\in [0,\infty)\times
[0, \infty)$, then $(t,s)= (t' , s')$.

So, let $\varphi : A_{t,s} \to A_{t',s'}$ be an isomorphism. Then $\varphi$ is linear and
unital. In particular, $\varphi (\lambda, 0) = (\lambda, 0)$ for every $\lambda \in \RR$.
Furthermore, we have
$$
\{ (0,x) \in A_{t,s} \, | \, x\in \RR^2 \} = \{ u \in A_{t,s}\, | \, u^2 \in \RR  \ \,
{\rm and}\ \, u\not\in \RR  \}\cup\{0\}.
$$
It follows that
$$
\varphi (\lambda , x) = (\lambda , Qx )
$$
for some linear map $Q : \RR^2 \to \RR^2$. From
$$
(\lambda^2 - \| Qx\|^2 , 2\lambda Qx) = (\lambda , Qx)^2 = (\varphi (\lambda , x))^2 
$$
$$
= \varphi ((\lambda, x)^2) = \varphi (\lambda^2 - \| x\|^2 , 2 \lambda x) =
(\lambda^2 - \|x\|^2 , 2 \lambda Qx)
$$
we get that $\| Qx\|^2 = \| x\|^2$ for every $x\in \RR^2$. Thus, $Q$ is orthogonal. 
The equation
$$
\varphi ((\lambda, x)\, (\mu, y)) = \varphi(\lambda, x)\, \varphi (\mu , y)
$$
can be rewritten as
$$
(\lambda \mu - \langle x,y \rangle + t |x\ y| ,\lambda Qy
+ \mu Qx + s |x\ y| Qe_1 )
$$
$$
= (\lambda \mu - \langle x,y \rangle + t' (\det Q) \, |x\ y| ,\lambda Qy
+ \mu Qx + s' (\det Q) \,  |x\ y|e_1 ).
$$
We conclude that $t=t' \det Q$ and $sQe_1 = s' (\det Q) e_1$. Applying the fact that
$| \det Q| =1$ and $\| Qe_1 \| =
\| e_1 \| =1$ we get $| t| = |t'|$ and $|s | = |s'|$. As all $t,t',s,s'$ are nonnegative,
we have $t=t'$ and $s=s'$, as desired.
\end{proof}

It follows directly from Corollary \ref{Cnn} that $A_{t,s}$ is nicely normed if and only if $t=0$.
So, the above statement shows that there is a natural bijection between $[0, \infty)$ and 
isomorphism classes of 3-dimensional nicely normed algebras.

The next result owes a lot to the paper \cite{Die}
classifying 4-dimensional real quadratic division algebras.
Our approach 
covers a more general class of real algebras. It is self-contained and completely elementary
using just simple linear algebra tools. 

We identify linear maps on $\RR^3$ with $3\times 3$ real matrices. Let $M_3$ denote the set
of all $3\times 3$ real matrices. For $(T,u),(T',u') \in M_3 \times \RR^3$ we write $(T,u)
\sim (T',u')$ if and only if
there exists an orthogonal $3\times 3$ matrix $Q$ such that $T' = (\det Q) QTQ^T$ and $u'=(\det Q)
Q u$.  It is clear that $\sim$ is an equivalence relation on $M_3 \times \RR^3$.
The set of equivalence classes will be denoted by $(M_3 \times \RR^3)/\sim$. 

For $T\in M_3$ and $u\in \RR^3$
we denote by $A_{T,u}$ the 4-dimensional
algebra $A_{T,u}= \RR \times \RR^3$ with the multiplication given by
$$
(\lambda , x) \, (\mu , y) = (\lambda \mu - \langle x,y \rangle + (x,y,u) ,\lambda y
+ \mu x + T(x\times y) ).
$$
As in the 3-dimensional case one can easily verify that $A_{T,u}$ is
a locally complex algebra. We will show that each 4-dimensional locally complex
algebra $A$ is isomorphic to $A_{T,u}$ for some $(T,u)\in M_3 \times \RR^3$
and that $A_{T,u}$ and $A_{T', u'}$ are isomorphic if and only if $(T,u) \sim (T' , u')$.
In other words, we will prove the following.

\begin{theorem}\label{bv2}
The map $(T,u) \mapsto A_{T,u}$, $T\in M_3$, $u\in \RR^3$, induces a bijection between
$(M_3 \times \RR^3)/\sim$ and isomorphism classes of 4-dimensional
locally complex algebras.
\end{theorem}

\begin{proof} We will first show that each $4$-dimensional locally complex algebra $A$ is isomorphic to
$A_{T,u}$ for some $(T,u)\in M_3 \times \RR^3$. 
It is a straightforward consequence of Lemma \ref{Tlc}\,(v) that $A$ is isomorphic
to $\RR \times \RR^3$ with the multiplication given by
$$
(\lambda , x) \, (\mu , y) = (\lambda \mu - \langle x,y \rangle ,\lambda y
+ \mu x ) + S(x_1 y_2 - x_2 y_1 , x_1 y_3 - x_3 y_1 , x_2 y_3 - x_3 y_2 )
$$
for some linear map $S : \RR^3 \to \RR \times \RR^3$. Observe that 
$S : \RR^3 \to \RR \times \RR^3$ can be decomposed into a direct sum of a linear functional on $\RR^3$
and an endomorphism on $\RR^3$. Recall that every linear functional on $\RR^3$ can be represented in a unique way
as an inner product with a fixed vector in $\RR^3$. Finally, observe that the coordinates of the vector
$(x_1 y_2 - x_2 y_1 , x_1 y_3 - x_3 y_1 , x_2 y_3 - x_3 y_2 )$ are up to a permutation and a multiplication
by $\pm 1$ the coordinates of the vector product $x \times y$. Thus, $A$ is 
isomorphic
to $\RR \times \RR^3$ with the multiplication given by
$$
(\lambda , x) \, (\mu , y) = (\lambda \mu - \langle x,y \rangle + (x,y,u) ,\lambda y
+ \mu x + T(x \times y)) 
$$
for some $u\in \RR^3$ and some endomorphism $T$ of $\RR^3$. Hence, $A$ is isomorphic to $A_{T,u}$, as
desired.

Assume now that
$A_{T,u}$ and $A_{T', u'}$ are isomorphic for some $(T,u),(T' , u')\in M_3\times
\RR^3$. We have to show that $(T,u)\sim (T' , u')$.

So, let $\varphi : A_{T,u} \to A_{T',u'}$ be an isomorphism. Exactly in the same way as
in the 3-dimensional case we show that
$$
\varphi (\lambda , x) = (\lambda , Qx )
$$
for some orthogonal $3\times 3$ matrix $Q$. 
The equation
$$
\varphi ((\lambda, x)\, (\mu, y)) = \varphi(\lambda, x)\, \varphi (\mu , y)
$$
can be rewritten as
$$
(\lambda \mu - \langle x,y \rangle + (x,y,u) ,\lambda Qy
+ \mu Qx + QT(x\times y) )
$$
$$
= (\lambda \mu - \langle x,y \rangle + (Qx,Qy,u') ,\lambda Qy
+ \mu Qx + T' (Qx \times Qy) ).
$$
We conclude that 
$$
(x,y,u)= (Qx, Qy , u') 
$$
and
$$
QT(x\times y) = T' (Qx \times Qy)
$$
for all $x,y \in \RR^3$. As $Q$ is orthogonal we have $Q(x\times y) = (\det Q ) (Qx
\times Qy)$, and consequently,
$$
(x,y,u) = (\det Q)\, (x,y, Q^T u') \ \ \ {\rm and}\ \ \  QT(x\times y) = (\det Q)\,
T'Q (x\times y),\ \ \ x,y \in \RR^3.
$$
It follows that $u' = (\det Q) Q u$ and $T' = (\det Q) QTQ^T$, as desired.

Finally, if $(T,u) \sim (T' ,u')$ for some $T,T' \in M_3$ and $u,u' \in \RR^3$ then 
there exists an orthogonal $3\times 3$ matrix $Q$ such that $T' = (\det Q) QTQ^T$ and $u'=(\det Q)
Q u$. It is then straightforward to check that the map $\varphi : A_{T,u} \to A_{T' , u'}$ 
defined by $\varphi (\lambda, x) = (\lambda , Qx)$, $(\lambda, x) \in A_{T,u}$,
is an
isomorphism.
\end{proof}

It is rather easy to verify that $A_{T,u}$ is nicely normed if and only if $u=0$. We will next show that
$A_{T,u}$ is a division algebra if and only if $\langle Tx , x \rangle \not=0$ for each nonzero $x\in \RR^3$
(that is, the quadratic form $q(x) = \langle Tx , x \rangle$ is either positive definite, or negative definite).
Indeed, assume first that $A_{T,u}$ is not a division algebra. Then 
$$
(\lambda \mu - \langle x,y \rangle + (x,y,u) ,\lambda y
+ \mu x + T(x\times y) ) = 0
$$
for some nonzero $(\lambda , x) , (\mu , y) \in A_{T,u}$. In particular,
$$
T(x\times y) = -\lambda y -\mu x.
$$
Set $z = x \times y$. We have $z\not=0$, since otherwise $x$ and $y$ are linearly dependent and therefore
\begin{itemize}
\item either $\lambda = 0$ and then $\langle x,y \rangle = 0$ and $\mu x=0$ which further yields that  
$(\lambda , x) = 0$ or $(\mu , y ) = 0$, a contradiction; or
\item $\mu=0$ which yields a contradiction in exactly the same way; or
\item $\lambda\not=0$ and $\mu\not=0$ and
then $y = -\mu \lambda^{-1} x$ and $\lambda \mu = \langle x, y \rangle$ yield
$0< \lambda^2 = -\langle x , x \rangle \le 0$, a contradiction.
\end{itemize}
Hence, $z\not=0$ and because $z$ is orthogonal to both $x$ and $y$ we have $\langle Tz , z \rangle =0$.

To prove the other direction we assume that there exists $z\in \RR^3$ with $\| z \| =1$ and $\langle Tz,
z \rangle =0$. Then $Tz = -tw$ for some real number $t$ and some 
$w\in \RR^3$ with $w\perp z$ and $\| w \| =1$. There is a unique $v\in \RR^3$ such that $z = w \times v$ and $v\perp w$. 
Set $s=-(w, v, u)$.
Then
$(0,w)$ and $(t, v- sw)$ are nonzero elements of $A_{T,u}$ whose product is equal to zero. Hence,
$A_{T,u}$ is not a division algebra, as desired.

Following Dieterich's idea \cite{Die} we will now disscuss  a geometric interpretation of the classification of 4-dimensional
locally complex algebras. 
Let us start with a simple observation concerning $3\times3$ skew-symmetric matrices.
If $x,y \in \RR^3$ are any two vectors such that $x\times y = (c_1 , c_2 , c_3)$, then
$$
R= \left[ \begin{array}{ccc} 0 & c_3 & -c_2 \\ -c_3 & 0 & c_1 \\
c_2 & -c_1 & 0 \end{array} \right] = xy^T - yx^T,
$$
where $x$ and $y$ are represented as $3\times 1$ matrices. If $Q$ is any orthogonal matrix, then $QRQ^T =
(Qx)(Qy)^T - (Qy)(Qx)^T$. As $Qx \times Qy = (\det Q)\, Q(x\times y)$, we have
$$
Q \left[ \begin{array}{ccc} 0 & c_3 & -c_2 \\ -c_3 & 0 & c_1 \\
c_2 & -c_1 & 0 \end{array} \right] Q^T = \left[ \begin{array}{ccc} 0 & d_3 & -d_2 \\ -d_3 & 0 & d_1 \\
d_2 & -d_1 & 0 \end{array} \right],
$$
where
$$
\left[ \begin{array}{ccc} d_1 \\ d_2 \\
d_3 \end{array} \right] = (\det Q) \, Q 
\left[ \begin{array}{ccc} c_1 \\ c_2 \\
c_3 \end{array} \right].
$$
If we choose $Q\in SO(3)$ such that 
$$
\left[ \begin{array}{ccc} 0 \\ 0 \\ \sqrt{c_{1}^2 + c_{2}^2 + c_{3}^2} 
\end{array} \right] =  Q 
\left[ \begin{array}{ccc} c_1 \\ c_2 \\
c_3 \end{array} \right],
$$
then
$$
QRQ^T = \left[ \begin{array}{ccc} 0 & d & 0 \\ -d & 0 & 0 \\
0 & 0 & 0 \end{array} \right],
$$
where $d= \sqrt{c_{1}^2 + c_{2}^2 + c_{3}^2}$. In particular, $d = \| R \|$.

Any $3\times 3$ matrix $T$ can be uniquely decomposed into its symmetric and skew-symmetric part, $T=P+R$,
$P= (1/2)(T + T^T)$, $R = (1/2)(T - T^T)$. If $T' = (\det Q) QTQ^T$ and $T' = P' +R'$ with $P'$ symmetric and $R'$ skew-symmetric,
then $P' = (\det Q) QPQ^T$ and $R' = (\det Q) QRQ^T$. We will say that $A_{T,u}$ is of rank 3,2,1,0, respectively,
if the symmetric part $P$ of $T$ is of rank 3,2,1,0, respectively. By the previous remark, two isomorphic
algebras $A_{T,u}$ have the same rank.

Let us start with algebras $A_{T,u}$ of rank 3. We have two possibilities: either all eigenvalues of $P= T+T^T$
have the same sign, or $P$ has both positive and negative eigenvalues. In the first case we will say that
$A_{T,u}$ is an ellipsoid locally complex algebra of dimension 4, while in the second case we call
$A_{T,u}$ a hyperboloid locally complex algebra of dimension 4. As we are interested in isomorphism classes we can
use the fact that $A_{T,u}$ is isomorphic to $A_{-T,u}$ to restrict our attention to the case when all the eigenvalues
of $P$ are positive (the ellipsoid case) or to the case when two eigenvalues of $P$ are positive and one is negative
(the hyperboloid case). Once we have done this restriction two algebras $A_{T,u}$ and $A_{T' , u'}$ of the above types
are isomorphic if and only if $T' = QTQ^T$ and $u' = Qu$ for some $Q\in SO(3)$.

To consider isomorphism classes of hyperboloid locally complex algebras of dimension 4
(a 4-dimensional locally complex algebra is hyperboloid if it is isomorphic to some
hyperboloid algebra $A_{T,u}$)
we set $\tau = \{ \delta \in \RR^3 \, | \, \delta_1 \ge \delta_2 > 0 > \delta_3 \}$ and
$\kappa = \tau \times \RR^3 \times \RR^3$. The elements of $\kappa$ will be called configurations.
Each configuration consists of a hyperboloid $H_\delta = \{ x\in \RR^3 \, | \, \langle \Delta_\delta x,x \rangle =1\}$ 
(a hyperboloid in principal axis
form) and a pair of points. Here, $\Delta_\delta$ is the diagonal matrix with the diagonal entries: $\delta_1 ,
\delta_2 , \delta_3$. The symmetry group of the hyperboloid $H_\delta$ is defined to be $G_\delta = \{ Q \in SO(3) \, | \, 
Q\Delta_\delta Q^T = \Delta_\delta \}$ (the requirement that $\det Q = 1$ tells that we allow only symmetries that
preserve the orientation). Note that this symmetry group consists of 4 elements whenever $\delta_1 > \delta_2$.
Namely, in this case the symmetry group consists of the identity and all diagonal matrices with two eigenvalues -1
and one eigenvalue 1. The symmetry group is infinite if and only if the hyperboloid $H_\delta$ is circular, that is,
$\delta_1 = \delta_2$. Two configurations $(\delta , u, c)$ and $(\delta' , u' ,c')$ are said to be equivalent,
$(\delta , u, c) \equiv (\delta' , u' ,c')$, if
and only if their hyperboloids coincide and their pairs of points lie in the same orbit under the operation of the symmetry
group of the hyperboloid, that is, if and only if
$\delta = \delta'$ and $(u' , c') = (Qu, Qc)$ for some $Q\in G_\delta$. We denote by 
$\kappa/ \equiv$ the set of equivalence classes of $\kappa$. We have a natural bijection between $\kappa/\equiv$
and the set of equivalence classes of hyperboloid locally complex algebras of dimension 4. Indeed, 
the bijection is induced by the map
$$
(\delta , u, c) \mapsto A_{\Delta_\delta + R_c,u}
$$
where
$$
\Delta_\delta + R_c = \left[ \begin{array}{ccc} \delta_1 & c_3 & -c_2 \\ -c_3 & \delta_2 & c_1 \\
c_2 & -c_1 & \delta_3 \end{array} \right].
$$
Clearly, $A_{\Delta_\delta + R_c,u}$ is a hyperboloid locally complex algebra.
We have to show that each hyperboloid algebra $A_{T, v}$ is isomorphic to some $A_{\Delta_\delta + R_c,u}$ and that
$A_{\Delta_\delta + R_c,u}$ and $A_{\Delta_{\delta'} + R_{c'} ,u'}$ are isomorphic if and only if
$(\delta , u, c) \equiv (\delta' , u' ,c')$. The second statement is trivial. To verify the first one
we write $T=P+R$ with $P$ symmetric with two positive eigenvalues and $R$ skew-symmetric. Then there exists
$Q\in SO(3)$ such that $QPQ^T = \Delta_\delta$ for some $\delta \in \tau$. We have $QRQ^T = R_c$ for some
$c\in \RR^3$. Set $u=Qv$ to complete the proof.

In a similar fashion we can consider isomorphism classes of ellipsoid locally complex algebras of dimension 4. Note that
a locally complex algebra $A_{T,u}$ is a division algebra if and only if it is an ellipsoid algebra. As above we can
consider configurations which consist of an ellipsoid in principal axis form and a pair of points. To each such configuration 
there corresponds a 4-dimensional real division algebra and this correspondence induces a bijection between the equivalence
classes of configurations (the equivalence being defined via the symmetry group of the ellipsoid) and the isomorphism 
classes of 4-dimensional real quadratic division algebras. We omit the details that can be found in \cite{Die}. It is clear that
locally complex algebras of rank 2 are either elliptic cylinder algebras or hyperbolic cylinder algebras. We leave the
details to the reader. In the same way one can classify also isomorphism classes of
locally complex algebras of rank 1. Let us conclude with the detailed disscussion on 4-dimensional locally complex 
algebras of rank 0. By $e_3$ we denote $e_3 = (0,0,1) \in \RR^3$.
We define an equivalence relation on the set $[0, \infty) \times \RR^3$ as follows:
$(d,u), (d',u')\in [0, \infty) \times \RR^3$ are said to be equivalent, $(d,u) \equiv (d', u')$, if either
\begin{itemize}
\item $d=d' =0$ and $\| u \| = \| u' \|$; or 
\item $d=d' >0$, $\| u \| = \| u '\|$, and
$\langle u, e_3 \rangle  =  \langle u' , e_3 \rangle $.
\end{itemize}
Note that the equivalence class of $(d,u)\in [0 , \infty) \times \RR^3$ with $d>0$ contains infinitely
many elements if $u$ and $e_3$ are linearly independent, and is a singleton when $u$ is a  scalar multiple of $e_3$.
There is a natural bijection between the
isomorphism classes of 4-dimensional locally complex algebras of rank 0 and the set $([0, \infty) \times \RR^3)/\equiv$.
The bijection is induced by the map from $[0, \infty) \times \RR^3$ which maps the pair $(d, u)$, $d\ge 0$, $u\in \RR^3$,
into $A_{T_d , u}$ with
$$
T_d= \left[ \begin{array}{ccc} 0 & d & 0 \\ -d & 0 & 0 \\
0 & 0 & 0 \end{array} \right].
$$
Obviously, $A_{T_d , u}$ is a locally complex algebra of rank 0 and one can easily verify that each
4-dimensional locally complex algebra of rank 0 is isomorphic to some $A_{T_d , u}$. It remains to show that
$A_{T_d , u}$ and $A_{T_{d'}, u'}$ are isomorphic if and only if $(d,u) \equiv (d' , u')$. So, assume that
$A_{T_d , u}$ and $A_{T_{d'}, u'}$ are isomorphic for some $(d,u), (d' , u') \in [0, \infty) \times \RR^3$.
Then there exists an orthogonal matrix $Q$ such that $T_{d'} = (\det Q) QT_d Q^T $ and  
$u' = (\det Q) Qu$. In particular, $d' = \| T_{d'} \| = \| T_d \| = d$ and $\| u' \| = \| u \|$. If $d=0$, then
$d' = 0$, and hence, $(d,u) \equiv (d' , u')$ in this special case. Therefore we may assume that $d=d' >0$. From
$T_{d'} = (\det Q) QT_d Q^T$ we conclude that $Qe_3 = (\det Q) e_3$. Consequently,
$$
\langle u' , e_3 \rangle  = \langle (\det Q) Qu, (\det Q) Qe_3 \rangle  =  \langle u, e_3 \rangle .
$$
To prove the converse we assume that $(d,u) \equiv (d' , u')$. We have one of the two possibilities
and we will consider just the second one. So, assume that
$d=d' >0$, $\| u \| = \| u '\|$, and
$\langle u, e_3 \rangle  =  \langle u' , e_3 \rangle $.
Then there exists an orthogonal matrix $Q$ such that $Qe_3 = e_3$ and $Qu = u'$. 
The orthogonal complement of $e_3$ and $u$ is one-dimensional (if $e_3$ and $u$ are linearly independent) or
two-dimensional (if $e_3$ and $u$ are linearly dependent). We have a freedom to choose the action of $Q$ on
the orthogonal complement of $e_3$ and $u$ (of course, up to the requirement that $Q$ is an orthogonal matrix).
In particular, we can choose $Q$ in such a way that $\det Q =1$.
 It follows
that $T_{d'} =  QT_d Q^T $ and  
$u' =  Qu$, as desired.

\section{Super-alternative locally complex algebras} \label{Secsuper}

Let us call an algebra $A$ a {\em super-alternative algebra} if it is $\mathbb Z_2$-graded, $A = A_0\oplus A_1$, and the alternativity conditions \eqref{alt} hold 
for all its   homogeneous elements. Equivalently,
\begin{equation}\label{alte}
u^2 x = u(ux),\,\, xu^2 = (xu)u \quad\mbox{for all $u\in A_i$, $i\in \ZZ_2$, $x\in A$,} 
\end{equation}
or, in the linearized form,
\begin{eqnarray}\nonumber
&&(uv + vu) x = u(vx) + v(ux),\\  \label{alte0}&&x(uv+vu) = (xu)v + (xv)u \quad\mbox{for all $u,v\in A_i$,  $i\in \ZZ_2$ , $x\in A$.} 
\end{eqnarray}
   
 The notion of a super-alternative algebra  should not be confused with the notion of an {\em alternative superalgebra}. The latter is defined through the alternativity of the Grassmann envelope of $A$. It turns out that nontrivial examples of alternative superalgebras exist only very exceptionally: prime alternative superalgebras of characteristic different from $2$ and $3$ are either associative or their odd part is zero \cite{SZ}. As we shall see, super-alternative algebras are more easy to find.

 Throughout this section {\em $A$ will be a super-alternative locally complex algebra}.
Our goal is to to classify all such algebras $A$. Obvious examples are   $\RR$, $\CC$, $\HH$, and $\OO$, as we can always take  the trivial $\ZZ_2$-grading (the odd part is $0$). Further,
one can check by a straigtforward calculation that if $\AAA_{n-1}$ is an alternative algebra, then every  $u\in (\AAA_{n-1}\times 0)\cup (0\times \AAA_{n-1})$ satisfies 
\eqref{alte} for every  $x\in \AAA_n$. Therefore, $\CC$, $\HH$, $\OO$, and $\SSS$ are super-alternative algebras with respect to the natural $\ZZ_2$-grading  mentioned in Section \ref{Prel}. Of course, the important information for us in this context is that $\SSS$ is also a super-alternative locally complex algebra. As we shall see, besides  $\RR$, $\CC$, $\HH$, $\OO$ and $\SSS$  only two more algebras must be added to the complete list of such algebras.

We continue by recording several simple but useful observations. First,  the following special case of \eqref{alte0} will be often used: 
 
\smallskip
(a)  If $u,v\in A_i$, $i\in\ZZ_2$, are such that $uv+vu=0$, then $u(vx) =- v(ux)$ and $(xu)v=-(xv)u$ for all $x\in A$.
 \smallskip

  If $v\in A_1$, then $v^2 \in A_0$; on the other hand, $v^2 = \lambda v +\mu$ for some $\lambda,\mu\in \RR$. Since $v\notin A_0$, we must have $\lambda =0$ and hence $v^2 =\mu\in\RR$. Since $A$ is locally complex, it follows that $\mu< 0$ if $v\ne 0$. Thus, we have
\smallskip

(b) If $0\ne v\in A_1$, then there is $\alpha\in \RR$ such that $(\alpha  v)^2 = -1$. 

 \smallskip
 Let $u\in A_0$ and $v\in A_1$ be such that $u^2= v^2=-1$. 
Using Lemma \ref{L0} we have $uv + vu \in\RR\cap A_1 =0$. Therefore $v(uv) = - v(vu) = -v^2 u = u$. Next, $(uv)v = uv^2 = -u$. Similarly we see that $(uv)u = - u(uv) = v$. Finally, using
(a) we get $(uv)(uv) = -(uv)(vu) = v((uv)u) = v^2 = -1$.
We have proved: 
 
 \smallskip
 
 (c) If $u\in A_0$ and $v\in A_1$ are such that $u^2= v^2=-1$, then $uv = -vu$, $v(uv) = - (uv)v = u$, $(uv)u = - u(uv) = v$, and $(uv)^2 = -1$.
 \smallskip

Let $u$ be a  homogeneous element and  suppose that $ux=0$ for some $x\in A$. If $u\ne 0$, then by multiplying this identity from the left by $u-t(u)$ it follows from \eqref{alte} that $n(u)x=0$, and hence $x=0$. Similarly,  $xu=0$  implies $x=0$ if $u\ne 0$. Thus:

\smallskip
(d) Homogeneous elements are not zero divisors.
 \smallskip

It is clear that our conditions on $A$ imply that $A_0$ is a locally complex alternative algebra. Theorem \ref{TFZ} therefore tells us that $A_0$ is isomorphic  to $\RR$, $\CC$, $\HH$, or $\OO$.
If $A_1=0$, then we get the desired conclusion that $A=A_0$ is one of the algebras from the expected list. 
 Without loss of generality we  may therefore assume that $A_1\ne 0$. Given $0\ne u\in A_1$, it follows from (d) that $x\mapsto ux$ is an injective linear map from $A_0$ into $A_1$; the same rule defines an injective linear map from $A_1$ into $A_0$. We may therefore conclude that 

\smallskip
(e)  $\dim A_0 = \dim A_1$.
 \smallskip

In particular we now know that a  super-alternative locally complex algebra must be finite dimensional. Moreover, its dimension can be only $1$, $2$, $4$, $8$, or $16$.

We shall now consider separately each of the four possibilities concerning $A_0$.

\begin{lemma} \label{case1}
If $A_0 \cong \RR$, then $A\cong\CC$. 
\end{lemma}

\begin{proof}  By (b) there is $i\in A_1$ with $i^2=-1$, and hence $A\cong\CC$ by (e).  
\end{proof}

\begin{lemma} \label{case2}
If $A_0 \cong \CC$, then $A\cong\HH$. 
\end{lemma}

\begin{proof} 
 We have $A_0 = \RR \oplus \RR i$ with $i^2=-1$. By (b) we may pick $j\in A_1$ such that $j^2=-1$.
 Setting $k=ij\in A_1$ it follows from (c) that  $A$ contains a copy of $\HH$. However, in view of (e) we actually have $A\cong\HH$.
\end{proof}

Let us now introduce another (an unexpected one for us) example of a super-alternative locally complex algebra. Let  $\TO$ be the $8$-dimensional algebra with basis $\{1,f_1,\ldots,f_7\}$ and multiplication table 
\begin{center}
\begin{small}
\begin{tabular}{|r|r|r|r|r|r|r|r|}
\hline
  & $f_1$ & $f_2$ & $f_3$ & $f_4$ & $f_5$ & $f_6$ & $f_7$\\\hline
  $f_1$ & $-1$ & $f_3$ & $-f_2$ & $f_5$ & $-f_4$ & $f_7$ & $-f_6$\\\hline
   $f_2$ & $-f_3$ & $-1$ & $f_1$ & $f_6$ & $-f_7$ & $-f_4$ & $f_5$\\\hline
     $f_3$ & $f_2$ & $-f_1$ & $-1$ & $f_7$ & $f_6$ & $-f_5$ & $-f_4$\\\hline
       $f_4$ & $-f_5$ & $-f_6$ & $-f_7$ & $-1$ & $f_1$ & $f_2$ & $f_3$\\\hline
         $f_5$ & $f_4$ & $f_7$ & $-f_6$ & $-f_1$ & $-1$ & $f_3$ & $-f_2$\\\hline
           $f_6$ & $-f_7$ & $f_4$ & $f_5$ & $-f_2$ & $-f_3$ & $-1$ & $f_1$\\\hline
             $ f_7$  & $f_6$ & $-f_5$ & $f_4$ & $-f_3$ & $f_2$ & $-f_1$ & $-1$\\\hline
\end{tabular}
\end{small}
\end{center}

\begin{lemma} \label{lto}
 $\TO$ is a super-alternative locally complex algebra with zero divisors and without alter-scalar elements (and hence $\TO\not\cong\OO$).
\end{lemma}

\begin{proof} The fact that $\TO$ is locally complex follows from Lemma \ref{Tlc}\,(v). Let $\TO_0$ be the linear span of $1,f_1,f_2,f_3$, and let $\TO_1$ be the linear span of $f_4,f_5,f_6,f_7$. Then $\TO$ becomes a superalgebra with the even part $\TO_0\cong\HH$. From the way we shall arrive at $\TO$ in the next proof it is not really surprising that  $\TO$ is super-alternative. But we used Mathematica for the actual checking that this is indeed true. Note that $(f_1-f_4)(f_3 - f_6) =0$, so that $\TO$ has zero divisors. Let $a\in \TO$  be such that $x^2 a = x(xa)$ for all $x\in \TO$.
From $(f_i + f_j)^2a= (f_i+f_j)((f_i + f_j)a)$, together with $f_i(f_ia) = f_j(f_ja) = -a$,  it follows  that $f_i(f_ja) + f_j(f_i a) = 0$ whenever $i\ne j$. Writing $a = \lambda_0 +\sum_{k=1}^7 \lambda_k f_k$ we thus have
\begin{equation}\label{fijk}
\sum_{k=1}^7 \lambda_k \Bigl(f_i(f_j f_k) + f_j(f_i f_k)\Bigr) = 0 \quad\mbox{whenever $i\ne j$.}
\end{equation}
Chosing $i=1$ and $j=4$ it follows that $\lambda_2=\lambda_3 =\lambda_6 =\lambda_7=0$. Chosing, for example, $i=2$ and $j=7$ we further get $\lambda_1 =\lambda_4=0$, and chosing 
$i=3$ and $j=4$ finally leads to $\lambda_5=0$. Therefore $a =\lambda_0$ is a scalar.
\end{proof}

\begin{lemma} \label{case3}
If $A_0 \cong \HH$, then $A\cong\OO$ or $A\cong\TO$. 
\end{lemma}

\begin{proof} 
Let $\{1,i,j,k\}$ be a basis of $A_0$ where these elements have the usual meaning. Pick $f\in A_1$ with $f^2 = -1$. Then $f$ anticommutes with $i,j,k$ by (c). It is clear that $\{f,if,jf,kf\}$ is a basis of $A_1$. We claim that all elements in this basis pairwise anticommute. It is easy to see that $f$ anticommutes with each of $if,jf,kf$. 
Using (a) repeatedly we obtain $(if)(jf)=-(i(jf))f=(j(if))f = -(jf)(if)$. Other identities can be checked analogously.

 Since $i(jf)\in A_1$, we have 
\begin{equation}\label{lambde}
i(jf) = \lambda_1 f + \lambda_2 if + \lambda_3 jf +\lambda_4 kf
\end{equation}
 for some $\lambda_i\in \RR$. From (a) we infer that $(i(jf))f = - (if)(jf)$. Similarly,
using (a) and (c) we get
$$
f(i(jf)) = -f((jf)i) = (jf)(fi) = -(jf)(if) = (if)(jf). 
$$ 
The last two identities show that $i(jf)$ anticommutes with $f$. Consequently, anticommuting \eqref{lambde} with $f$ it follows that $\lambda_1=0$. A similar arguing shows that 
$i(jf)$ anticommutes with both $if$ and $jf$, which leads to $\lambda_2 = \lambda_3=0$.  
Note that (c) implies that the squares of both $kf$ and $i(jf)$ are equal $-1$. But then $\lambda_4^2 =1$, i.e., $\lambda_4 = 1$ or $\lambda_4=-1$. If $\lambda_4=1$, i.e., $i(jf) = kf$,
 then we set $f_1=i$, $f_2=j$, $f_3=k$, $f_4=f$, $f_5=if$, $f_6=jf$, and $f_7=kf$. Using the information we have,  it is now just a matter of a routine calculation to verify that 
 $A\cong\TO$.
  Since we know that $\OO$ is a super-alternative locally complex algebra, the other possibility $\lambda_4=-1$ can lead only  to $A\cong \OO$. 
\end{proof}

The 16-dimensional analogue of $\TO$ is the algebra which we denote by $\TS$ and define as follows: if $\{1,f_1,\ldots,f_{15}\}$ is its basis, then the multiplication table is

\begin{center}
\begin{tiny}
\begin{tabular}{|r|r|r|r|r|r|r|r|r|r|r|r|r|r|r|r|r|}
\hline
& $f_1$& $f_2$& $f_3$& $f_4$& $f_5$& $f_6$& $f_7$& $f_8$& $f_9$&$f_{10}$& $f_{11}$& $f_{12}$& $f_{13}$& $f_{14}$& $f_{15}$\\\hline
 $f_1$& $-1$& $f_3$&$-f_2$& $f_5$& $-f_4$& $-f_7$& $f_6$& $f_9$& $-f_8$& $-f_{11}$& $f_{10}$&$-f_{13}$& $f_{12}$& $-f_{15}$& $f_{14}$\\\hline 
  $f_2$& $-f_3$& $-1$& $f_1$& $f_6$&$f_7$& $-f_4$& $-f_5$& $f_{10}$& $f_{11}$& $-f_8$& $-f_9$& $-f_{14}$& $f_{15}$& $f_{12}$& $-f_{13}$\\\hline 
   $f_3$& $f_2$& $-f_1$& $-1$& $f_7$& $-f_6$& $f_5$& $-f_4$&$f_{11}$& $-f_{10}$& $f_9$& $-f_8$& $f_{15}$& $f_{14}$& $-f_{13}$& $-f_{12}$\\\hline
    $f_4$&$-f_5$& $-f_6$& $-f_7$& $-1$& $f_1$& $f_2$& $f_3$& $f_{12}$& $f_{13}$& $f_{14}$& $-f_{15}$& $-f_8$& $-f_9$& $-f_{10}$& $f_{11}$\\\hline
     $f_5$& $f_4$& $-f_7$& $f_6$&$-f_1$& $-1$& $-f_3$& $f_2$& $f_{13}$& $-f_{12}$& $-f_{15}$& $-f_{14}$& $f_9$&$-f_8$& $f_{11}$& $f_{10}$\\\hline
      $f_6$& $f_7$& $f_4$& $-f_5$& $-f_2$& $f_3$& $-1$& $-f_1$& $f_{14}$& $f_{15}$& $-f_{12}$& $f_{13}$& $f_{10}$& $-f_{11}$& $-f_8$& $-f_9$\\\hline
       $f_7$& $-f_6$& $f_5$& $f_4$& $-f_3$& $-f_2$& $f_1$& $-1$& $f_{15}$& $-f_{14}$& $f_{13}$& $f_{12}$& $-f_{11}$& $-f_{10}$& $f_9$& $-f_8$\\\hline
        $f_8$& $-f_9$& $-f_{10}$& $-f_{11}$& $-f_{12}$& $-f_{13}$& $-f_{14}$& $-f_{15}$& $-1$&$f_1$& $f_2$& $f_3$& $f_4$& $f_5$& $f_6$& $f_7$\\\hline
         $f_9$& $f_8$& $-f_{11}$& $f_{10}$& $-f_{13}$& $f_{12}$& $-f_{15}$& $f_{14}$& $-f_1$& $-1$& $-f_3$& $f_2$& $-f_5$& $f_4$& $-f_7$& $f_6$\\\hline
          $f_{10}$& $f_{11}$& $f_8$& $-f_9$& $-f_{14}$& $f_{15}$& $f_{12}$& $-f_{13}$& $-f_2$& $f_3$& $-1$& $-f_1$& $-f_6$& $f_7$& $f_4$& $-f_5$\\\hline
           $f_{11}$& $-f_{10}$& $f_9$& $f_8$& $f_{15}$& $f_{14}$& $-f_{13}$& $-f_{12}$& $-f_3$& $-f_2$& $f_1$& $-1$& $f_7$& $f_6$& $-f_5$& $-f_4$\\\hline
            $f_{12}$& $f_{13}$& $f_{14}$& $-f_{15}$& $f_8$& $-f_9$& $-f_{10}$& $f_{11}$& $-f_4$& $f_5$& $f_6$& $-f_7$& $-1$& $-f_1$& $-f_2$& $f_3$\\\hline
             $f_{13}$& $-f_{12}$& $-f_{15}$& $-f_{14}$& $f_9$& $f_8$& $f_{11}$& $f_{10}$& $-f_5$& $-f_4$& $-f_7$& $-f_6$& $f_1$& $-1$& $f_3$& $f_2$\\\hline
              $f_{14}$& $f_{15}$& $-f_{12}$& $f_{13}$& $f_{10}$& $-f_{11}$& $f_8$& $-f_9$& $-f_6$& $f_7$& $-f_4$& $f_5$& $f_2$& $-f_3$& $-1$& $-f_1$\\\hline 
                $f_{15}$& $-f_{14}$& $f_{13}$& $f_{12}$& $-f_{11}$& $-f_{10}$& $f_9$& $f_8$& $-f_7$& $-f_6$& $f_5$& $f_4$&$-f_3$& $-f_2$& $f_1$& $-1$\\\hline
\end{tabular}
\end{tiny}
\end{center}
The proof of the next lemma is similar to that of Lemma \ref{lto}. Therefore we omit details.

\begin{lemma} \label{lts}
 $\TS$ is a super-alternative locally complex algebra  without alter-scalar elements (and hence $\TS\not\cong\SSS$).
\end{lemma}

The final lemma is similar to Lemma \ref{case3}, but the proof is somewhat more complicated. One of the problems that we have to face in this proof is that we do not have a complete freedom in the selection of an element playing the role of $f$ from the proof of Lemma \ref{case3}. While $f$ was an arbitrary element in $A_1$ with square $-1$,
now we shall have to find a special one.

\begin{lemma} \label{case4}
If $A_0 \cong \OO$, then $A\cong\SSS$ or $A\cong\TS$. 
\end{lemma}

\begin{proof}
Let $\{1,e_1,\ldots,e_7\}$ be a basis of $A_0$ whose multiplication table is given in Section \ref{Prel}. We begin with three claims needed for future reference.

\smallskip
{\sc Claim 1}: Let $i,j\in\{1,2,\ldots,7\}$, $i\ne j$. If $p\in A_1$, then $q = p + (e_i e_j)(e_i (e_jp))$ satisfies $(e_ie_j)q = - e_i(e_jq)$.
\smallskip
 
 Indeed, by \eqref{alte} we have $(e_ie_j)q = (e_ie_j)p - e_i(e_jp)$, while using (a) and \eqref{alte} we get
 \begin{align*}
 e_i(e_jq) &= e_i(e_jp) + e_i(e_j((e_i e_j)(e_i (e_jp)))) = e_i(e_jp) - e_i((e_ie_j)(e_j (e_i (e_jp))))\\
 &= e_i(e_jp) + (e_ie_j)(e_i(e_j (e_i (e_jp))) = e_i(e_jp) - (e_ie_j)(e_j(e_i (e_i (e_jp)))\\
  &= e_i(e_jp) + (e_ie_j)(e_j(e_j p)) =  e_i(e_jp) - (e_ie_j)p,
  \end{align*}
so that $(e_ie_j)q = - e_i(e_jq)$.

\smallskip
{\sc Claim 2}: Let $i,j,k\in\{1,2,\ldots,7\}$ be such that $e_i,e_j,e_ie_j,e_k$ are linearly independent, and  let  $s\in A_1$ be such that $(e_ie_j)s =- e_i(e_j s)$. Then $t= s + (e_ie_k)(e_i(e_ks))$  also satisfies 
 $(e_ie_j)t =- e_i(e_jt)$. 
 
 (Let us add that (a) implies   $t= s + (e_ke_i)(e_k(e_is))$,  and that 
  $(e_ie_j)z =- e_i(e_j z)$ is equivalent to $(e_je_i)z =- e_j(e_i z)$; the order of indices is thus irrelevant.)
\smallskip

  Indeed, by now already familiar arguing we have
 \begin{align*}
 (e_ie_j)t &= (e_ie_j)s + (e_ie_j)((e_ie_k)(e_i(e_ks)))  = (e_ie_j)s - (e_ie_k)((e_ie_j)(e_i(e_ks)))\\
 &= (e_ie_j)s + (e_ie_k)(e_i((e_ie_j)(e_ks))) = (e_ie_j)s - (e_ie_k)(e_i(e_k((e_ie_j)s)))\\
  &= -\bigl(e_i(e_js) - (e_ie_k)(e_i(e_k(e_i(e_js))))\bigr) = -\bigl( e_i(e_js) + (e_ie_k)(e_k(e_i(e_i(e_js))))\bigr)\\
 &= -\bigl( e_i(e_js) - (e_ie_k)(e_k(e_js))\bigr) =  -\bigl( e_i(e_js) + e_i(e_i ((e_ie_k)(e_k(e_js))))\bigr)\\
&= -\bigl( e_i(e_js) - e_i((e_ie_k) (e_i(e_k(e_js))))\bigr) = -\bigl( e_i(e_js) + e_i((e_ie_k) (e_i(e_j(e_ks))))\bigr)\\
&=-\bigl( e_i(e_js) - e_i((e_ie_k) (e_j(e_i(e_ks))))\bigr) = -\bigl( e_i(e_js) + e_i(e_j((e_ie_k)(e_i(e_ks))))\bigr)\\
&=- e_i(e_jt).
  \end{align*}

\smallskip
{\sc Claim 3}: Let $i,j,k\in\{1,2,\ldots,7\}$, $i\ne j$,  and  let $\e\in \RR$ and  $w\in A_1$ be such that $(e_ie_j)w =\epsilon e_i(e_j w)$. Set $u = e_kw$. If $k\in \{i,j\}$,
then  $(e_ie_j)u =\epsilon e_i(e_j u)$, and if $k\notin \{i,j\}$,
then  $(e_ie_j)u =-\epsilon e_i(e_j u)$.
\smallskip
 
 If $k\in \{i,j\}$, then we may assume $k = j$ without loss of generality. We have 
 $$
 (e_i e_j)(u) = (e_i e_j)(e_ jw)= - e_j ((e_i e_j) w) = -\e e_j (e_i (e_jw)) = \e e_i (e_j u).  
 $$
If $k\notin \{i,j\}$, then we have
\begin{align*}
 &(e_i e_j)(u) = (e_i e_j)(e_ kw)= - e_k ((e_i e_j) w) \\
 =& -\e e_k (e_i (e_jw)) =\e e_i(e_k(e_j w)) =- \e e_i (e_j u).  
\end{align*}

 After establishing these auxiliary claims, we now begin the actual proof by picking a nonzero $u\in A_1$. As mentioned above, an arbitrary chosen $u$ may not be the right choice, so we have to "remedy"  it. Let $v' = u + (e_1e_2)(e_1(e_2u))\in A_1$. By Claim 1, $v'$ satisfies  $(e_1e_2)v'=-e_1(e_2v')$.  If $v'=0$, then we have $(e_1e_2)u=e_1(e_2u)$. But then $v'' = e_3u$ satisfies
 $(e_1e_2)v''=-e_1(e_2v'')$ by Claim 3. 
 Thus, in any case there is a nonzero $v\in  A_1$ such that 
 $$(e_1e_2)v=- e_1(e_2v).$$
  Now consider
$w' = v + (e_1 e_4)(e_1(e_4v))$. By Claim 1 we have $(e_1e_4)w' = - e_1(e_4w')$, and
 by Claim 2 we have $(e_1e_2)w'=- e_1(e_2w')$. If $w'=0$, then $(e_1e_4)v= e_1(e_4v)$.
 But then $w'' = e_2v$ satisfies $(e_1e_2)w''=- e_1(e_2w'')$
 and $(e_1e_4)w'' = - e_1(e_4w'')$.
    Thus, there exists a nonzero $w\in A_1$ satisfying 
    $$(e_1e_2)w=-e_1(e_2w),\,\, (e_1e_4)w=- e_1(e_4w).$$
      We now repeat the same procedure with respect to $e_2$ and $e_4$. That is, we introduce $x'= w + (e_2e_4)(e_2(e_4w))$, and apply Claims 1 and 2 to conclude that
      $(e_1e_2)x'=-e_1(e_2x')$, $(e_1e_4)x'=- e_1(e_4x')$, and $(e_2e_4)x'=- e_2(e_4x')$. If $x'=0$, then $(e_2e_4)w = e_2(e_4w)$, and therefore Claim 3 tells us that
     $(e_1e_2)x''=-e_1(e_2x'')$, $(e_1e_4)x''=- e_1(e_4x'')$, and $(e_2e_4)x''=- e_2(e_4x'')$, where $x'' = e_1 w$.  In any case we have found a
       a nonzero $x\in A_1$ satisfying 
       $$(e_1e_2)x=- e_1(e_2x),\,\, (e_1e_4)x=- e_1(e_4x),\,\,(e_2e_4)x=- e_2(e_4x).$$ 
       Considering $y' = x + (e_3e_4)(e_3(e_4x))$ we see from Claim 2 that $(e_1e_4)y'=- e_1(e_4y')$ and $(e_2e_4)y'=- e_2(e_4y')$, while apparently we cannot conclude that also $(e_1e_2)y'=- e_1(e_2y')$. However, multiplying $(e_1e_2)x=- e_1(e_2x)$ from the left by $e_1$ we get
 $e_1( (e_1e_2)x)= e_2x$, which can be written as $e_1(e_3x) = - (e_1 e_3)x$. Therefore Claim 2 yields $e_1(e_3y') = - (e_1 e_3)y'$. Multiplying this from the left by $e_1$ we arrive at the desired identity $(e_1e_2)y'=- e_1(e_2y')$. Also,  $(e_3e_4)y'=- e_3(e_4y')$ holds by Claim 1. We still have to deal with the case where $y'=0$, i.e.,  $(e_3e_4)x= e_3(e_4x)$. The usual reasoning now does not work, since we do not have "enough room" to apply Claim 3. Thus, the final conclusion is that there exists
 a nonzero $y\in A_1$ such that
 $$
 (e_1e_2)y=-e_1(e_2y),\,\,(e_1e_4)y=- e_1(e_4y),\,\,(e_2e_4)y=- e_2(e_4y), \,\,(e_3e_4)y=\pm e_3(e_4y).
 $$
 In view of (b) we may assume without loss of generality that $y^2=-1$. Let us first consider the case where $(e_3e_4)y=e_3(e_4y)$. We set $f_8 = y$ and
 $f_i = e_i$, $f_{i+8}= f_i f_8$, $i=1,\ldots,7$. By standard calculations one can now verify that $A\cong \TS$; checking all details is  lengthy and tedious, but straigtforward. The other possibility where  $(e_3e_4)y=-e_3(e_4y)$ of course leads to  $A\cong \SSS$.
\end{proof}

All lemmas together yield our main result.

\begin{theorem}\label{MT}
 A  super-alternative  locally complex algebra is isomorphic to $\RR$, $\CC$, $\HH$, $\OO$, $\TO$, $\SSS$, or $\TS$.
\end{theorem}

\begin{remark}\label{rema1}
In the course of the proof we did not use the assumption that \eqref{alte} holds for all $u,x\in A_1$. Therefore we can replace the  super-alternativity assumption by a slightly milder one.
\end{remark}

This list reduces to Cayley-Dickson algebras under the additional assumption that there exist alter-scalar elements.

\begin{corollary}\label{MTc1}
 A  super-alternative  locally complex algebra containing alter-scalar elements is isomorphic to $\RR$, $\CC$, $\HH$, $\OO$, or $\SSS$.
\end{corollary}

\begin{corollary}\label{MTc2}
 A  super-alternative  locally complex algebra which contains  alter-scalar elements, but is not alternative, is isomorphic  to $\SSS$.
\end{corollary}

Let $A$ be an algebra, and let $x\in A$. The {\em annihilator} of $x$ is the space  Ann$(x) = \{y\in A\,|\, xy =0\}$.  If 
$A =\AAA_n$ is a Cayley-Dickson algebra, then the dimension of Ann$(x)$ is a multiple of $4$ \cite{Biss, Moreno}. Moreover, if $A = \AAA_4 =\SSS$, then the dimension of Ann$(x)$ is exactly $4$ for every zero divisor $x$ in $A$ \cite[Section 12]{Biss}. The algebras $\TO$ and $\TS$ do not have this property. It is easy to check that $x =f_1-f_4\in \TO$ has the $2$-dimensional annihilator spanned by $f_2 + f_7$ and
$f_3 - f_6$. Further, the dimension of the annihilator of $x = f_3 + f_{12}\in\TS$ is $6$; it is spanned by $f_1 + f_{14}$, $f_2  - f_{13}$, $f_{4} + f_{11}$, $f_{5} + f_{10}$, $f_6-f_9$, and
$f_7-f_8$. Thus, we have

\begin{corollary}\label{MTc3}
 Let $A$ be a  super-alternative  locally complex algebra which is not a division algebra. If the dimension of {\rm Ann}$(x)$ is $4$ for every zero divisor in $A$, then $A\cong \SSS$.
\end{corollary} 

One can check that 
$$
1\mapsto 1,\,\,
e_1\mapsto f_1,\,\,
e_2\mapsto f_2,\,\,
e_3\mapsto f_3,\,\,
e_4\mapsto f_{12},\,\,
e_5\mapsto -f_{13},\,\,
e_6\mapsto-f_{14},\,\,
e_7\mapsto -f_{15}
$$
defines an embedding of $\TO$ into $\SSS$. Thus, both $\OO$ and $\TO$ can be viewed as subalgebras of $\SSS$.
Chan and \DJ okovi\' c proved that $\SSS$ has $6$-dimensional subalgebras, which, however, are not contained in $8$-dimensional subalgebras of $\SSS$ \cite[Corollary 3.6, Theorem 8.1]{ChanDj}.
Accordingly, $\OO$ and $\TO$ do not have $6$-dimensional subalgebras. Further, $\SSS$
does not contain $5$-dimensional subalgebras \cite[Proposition 4.4]{ChanDj}.  This does not hold for $\TS$. For example, the linear span of $1$, $f_{1}+f_{14}$, $f_{3}-f_{12}$,  $f_{6}-f_{9}$, and $f_{7}-f_{8}$ is a $5$-dimensional subalgebra of $\TS$. Combining all these we get our final corollary.

\begin{corollary}\label{MTc4}
 Let $A$ be a  super-alternative  locally complex algebra. If $A$ has $6$-dimensional subalgebras, but does not have $5$-dimensional  subalgebras,  then $A\cong \SSS$.
\end{corollary}

{\bf Acknowledgement}. 
The authors are  grateful to the referee for  careful reading of the
paper and the resulting useful remarks.

\end{document}